\documentclass[11pt,a4paper]{amsart}

\usepackage{amsmath, amsthm, amssymb}

\newcommand{\ord}{\operatorname{ord}}
\newcommand{\Per}{\operatorname{Per}}

\newcommand{\bR}{\mathbb{R}}
\newcommand{\bP}{\mathbb{P}}
\newcommand{\bC}{\mathbb{C}} 
 
\newcommand{\bN}{\mathbb{N}}
\newcommand{\rd}{\mathrm{d}} 
\newcommand{\loc}{\operatorname{loc}} 
\newcommand{\supp}{\operatorname{supp}}
\newcommand{\Fix}{\operatorname{Fix}}
\newcommand{\Res}{\operatorname{Res}}
  
\numberwithin{equation}{section}
\theoremstyle{plain}
\newtheorem{theorem}{Theorem}[section]
\newtheorem{lemma}[theorem]{Lemma}
\newtheorem*{claim}{Claim} 

\newtheorem{mainth}{Theorem}

\theoremstyle{definition}
\newtheorem{definition}[theorem]{Definition}
\newtheorem{notation}[theorem]{Notation}

\newtheorem*{acknowledgement}{Acknowledgement}

\theoremstyle{remark}
\newtheorem{remark}[theorem]{Remark}
\newtheorem{fact}[theorem]{Fact}

\begin{document}  

\title[Equidistribution of rational functions]{Equidistribution of rational functions having a superattracting
periodic point towards the activity current and the bifurcation current}

\author[Y\^usuke Okuyama]{Y\^usuke Okuyama}
\address{
Division of Mathematics,
Kyoto Institute of Technology,
Sakyo-ku, Kyoto 606-8585 Japan.}
\email{okuyama@kit.ac.jp}


\date{\today}

\begin{abstract}
 We establish an approximation of the activity current $T_c$ 
 in the parameter space of a holomorphic family $f$ of rational functions
 having a marked critical point $c$
 by parameters for which $c$ is periodic under $f$, 
 i.e., is a superattracting periodic point.
 This partly generalizes a Dujardin--Favre theorem
 for rational functions having preperiodic points, and
 refines a Bassanelli--Berteloot theorem on a similar approximation
 of the bifurcation current $T_f$ of the holomorphic family $f$.
 The proof is based on a dynamical counterpart of this approximation.
\end{abstract}


\maketitle

\section{Introduction}\label{sec:intro}

The $J$-stable locus $S_f$ in a holomorphic family $f$ of rational functions
is open and dense in the parameter space, 
contains the quasiconformally stable locus of $f$ as an open and dense subset, and is
characterized by the non-activity of all the critical points if they are marked
\cite{MSS} (see also \cite[Chapter 4]{McMullen:renorm} and \cite{Lyubich83stability}). 
The $J$-{\itshape un}stable locus or the {\itshape bifurcation} locus 
$B_f$ of $f$ can be also studied from a pluripotential theoretical viewpoint. 
Our aim is to contribute to the study of the instability in
a holomorphic family of rational functions and the activity of
its marked critical point. We give an affirmative answer,
in the superattracting case, to a question on the removability 
of a seemingly technical assumption on the parameter space posed by 
Dujardin--Favre \cite[Theorem 4.2]{DujardinFavre08},
and refines a result due to Bassanelli--Berteloot 
\cite[Theorem 3.1 (1)]{BassanelliBerteloot09}. See also 
survey articles \cite{BertelootCIME} and \cite{DujardinSurvey}.
 
\subsection{Equidistribution towards the activity current $T_c$}

We say a mapping $f:\Lambda\times\bP^1\to\bP^1$ 
is a holomorphic family of rational functions on $\bP^1$
of degree $d>1$ over a connected complex manifold $\Lambda$
if $f$ is holomorphic and for every $\lambda\in\Lambda$, 
$f_{\lambda}:=f(\lambda,\cdot)$ is a rational function on $\bP^1$ of degree $d$, 
and say that $f$ has a {\itshape marked critical point} $c:\Lambda\to\bP^1$ 
if $c$ is holomorphic and for every $\lambda\in\Lambda$,
$c(\lambda)$ is a critical point of $f_{\lambda}$. 

For the details of pluripotential theory, we refer to 
\cite[Chapter III]{Demaillybook} and \cite[Part I]{Klimek91}.

\begin{definition}\label{th:divisor}
 Let $\phi,\psi$ be meromorphic functions on a connected complex manifold $M$.
 If $\phi\not\equiv\psi$ on $M$, then let $[\phi=\psi]$ be the 
 {\itshape current of integration} over the divisor defined by the equation
 $\phi=\psi$ on $M$: the Poincar\'e-Lelong formula asserts that
 when $\phi$ and $\psi$ are holomorphic, 
 $\rd\rd^c\log|\phi-\psi|=[\phi=\psi]$ (for $\rd\rd^c$,
 see Notation \ref{th:notation}). On the other hand,
 {\itshape by convention,} if $\phi\equiv\psi$, then we set 
 $[\phi=\psi]:=0$ as a $(1,1)$-current on $M$; we should be careful
 for this convention since in the case
 that $\phi\equiv 0$ and $\psi\not\equiv 0$ on $M$,
 $[\phi=0]+[\psi=0]=[\psi=0]$ might be not equal to $[(\phi\psi)=0]=0$.
\end{definition}

 For each $n\in\bN$, set 
 \begin{gather*}
 F_n(\lambda):=f_\lambda^n(c(\lambda))\quad\text{on }\Lambda .
 \end{gather*}

\begin{definition}[the currents $\Per_c(n)$ and $\Per_c^*(n)$]\label{th:zerodivisor}
 Following \cite[Definition 4.1]{DujardinFavre08}, for every $n\in\bN$, set
 \begin{gather}
 \Per_c(n):=[F_n=c]\quad\text{on }\Lambda.\label{eq:superattdivisor}
 \end{gather}
 Moreover, for each $n\in\bN$,
 let $X_n$ be the closure in $\Lambda$ of
 $\supp[F_n=c]\setminus(\bigcup_{m\in\bN:\, m|n\text{ and }m<n}\supp[F_m=c])$,
 which is also an analytic subset in $\Lambda$ and whose irreducible components
 are those of $\supp\Per_c(n)$. Denoting by $[A]$ 
 the current of integration over an analytic variety $A$ in $\Lambda$, we set
\begin{gather}
 \Per_c^*(n):=\sum_V
 (\ord_V(\Per_c(n)))\cdot [V],\label{eq:exactsuperatt}
\end{gather}
 where the sum ranges over all irreducible components $V$ of $X_n$.
\end{definition}

\begin{definition}\label{th:active}
 Let $\omega$ be the Fubini-Study area element on $\bP^1$
 normalized as $\omega(\bP^1)=1$. To the marked critical point $c$ of $f$, 
 we can associate the {\itshape activity current} 
 \begin{gather}
 T_c:=\lim_{n\to\infty}\frac{F_n^*\omega}{d^n}\quad\text{ as a }(1,1)\text{-current on }\Lambda.\label{eq:activity}
 \end{gather} 
 The proof of the convergence of the right hand side
 is due to \cite[Proposition-Definition 3.1]{DujardinFavre08}
 (see also Remark \ref{th:detail}). 
 The support of $T_c$ coincides with the {\itshape activity locus}
 \begin{gather*}
 A_c:=\{\lambda\in\Lambda:\{F_n:n\in\bN\}\text{ is not normal at }\lambda\}
 \end{gather*}
 associated to $c$ (\cite[Theorem 3.2]{DujardinFavre08}).
\end{definition}

The following is our principal result: the convergence \eqref{eq:Levin}
partially generalizes
Dujardin--Favre \cite[Theorem 4.2]{DujardinFavre08} by removing their technical
assumption in our superattracting case. The foundational case that
$f(\lambda,z)=z^d+\lambda$ and $c\equiv 0$ on $\Lambda=\bC$
was due to Levin \cite{Levin90}.  

\begin{mainth}\label{th:periodicbif}
Let $f:\Lambda\times\bP^1\to\bP^1$ 
be a holomorphic family of rational functions on $\bP^1$
of degree $d>1$ over a connected complex manifold $\Lambda$ 
having a marked critical point $c:\Lambda\to\bP^1$.
Then
\begin{gather}
 \lim_{n\to\infty}\frac{\Per_c(n)}{d^n+1}=T_c\quad\text{as currents on }\Lambda,\quad\text{and}\label{eq:Levin}\\
 \lim_{n\to\infty}\frac{\Per_c^*(n)}{d^n+1}=T_c\quad\text{as currents on }\Lambda.\label{eq:exact}
\end{gather}
\end{mainth}

\begin{remark}\label{th:trivialcase}
 Both \eqref{eq:Levin} and \eqref{eq:exact} hold 
 even if $F_n\equiv c$ on $\Lambda$ for some $n\in\bN$:
 for, in this case, $\#\{F_n:n\in\bN\}<\infty$,
 so $\lim_{n\to\infty}\Per_c(n)/(d^n+1)=\lim_{n\to\infty}\Per_c^*(n)/(d^n+1)=0$ as currents on $\Lambda$
 and also $A_c=\emptyset$. The latter implies
 $T_c=0$ on $\Lambda$ since $\supp T_c\subset A_c$. 
\end{remark}

Our proof of Theorem \ref{th:periodicbif} relies on a dynamical counterpart of
this result, and is simpler than Dujardin--Favre's argument, which relies on
a delicate classification \cite[Theorem 4]{DujardinFavre08}
of non-active parameters.

\subsection{Equidistribution towards the bifurcation current $T_f$} 
\label{sec:bifurcationintro}
Let $f:\Lambda\times\bP^1\to\bP^1$ be a 
holomorphic family of rational functions of degree $d>1$
over a connected complex manifold $\Lambda$.

For every $\lambda\in\Lambda$, let $L(f_{\lambda})$ be the {\itshape Lyapunov
exponent} of $f_{\lambda}$ with respect to the unique maximal entropy measure
of $f_{\lambda}$. The function $\Lambda\ni \lambda\mapsto L(f_{\lambda})\in\bR$
is positive, continuous, and plurisubharmonic on $\Lambda$.
\begin{definition}[{DeMarco \cite[Theorem 1.1]{DeMarco03};
see also Pham \cite{Pham05} and Dinh--Sibony \cite[\S 2.5]{DSsurvey}}]
The {\itshape bifurcation current} $T_f$ on $\Lambda$ of $f$ is 
defined by
 \begin{gather*}
 T_f:=\rd\rd^c_{\lambda}L(f_{\lambda})\quad\text{as a }(1,1)\text{-current on }\Lambda.
 \end{gather*}
\end{definition}
 Taking a finitely-sheeted possibly ramified covering 
 of $\Lambda$ if necessary, we can assume that
 there are marked critical points $c_1,\ldots,c_{2d-2}:\Lambda\to\bP^1$ 
 of $f$ such that for every $\lambda\in\Lambda$, 
 $c_1(\lambda),\ldots, c_{2d-2}(\lambda)$
 are all the critical points of $f_{\lambda}$  
 taking into account their multiplicities. Then by DeMarco's formula
 \cite[Theorem 1.4]{DeMarco03} (see also Remark \ref{th:detail}), 
 $T_f$ is decomposed as
 \begin{gather}
 T_f=\sum_{j=1}^{2d-2}T_{c_j}.\label{eq:DeMarco}
 \end{gather}

\begin{definition}[a periodic point having the exact period]
 Fix $n\in\bN$ and $\lambda\in\Lambda$. A fixed point
 $w\in\bP^1$ of $f_{\lambda}^n$ is a periodic point of $f_{\lambda}$
 having the {\itshape exact period} $n$
 if for every $m\in\bN$ satisfying $m|n$ and
 $m<n$, $f_{\lambda}^m(w)\neq w$. Let $\Fix^*(f_{\lambda}^n)$
 be the set of all periodic points of $f_{\lambda}$ having
 the exact period $n$.
\end{definition}

For each $n\in\bN$, the holomorphic family $f$ induces
the {\itshape multiplier polynomial}
$p_n^*(\lambda,w)=p_{f,n}^*(\lambda,w)$ on $\Lambda\times\bC$, which satisfies
that $(\lambda,w)\mapsto p_n^*(\lambda,w)$ 
is a holomorphic function on $\Lambda\times\bC$,
that for each $\lambda\in\Lambda$, $p_n^*(\lambda,\cdot)$ is a polynomial on $\bC$,
and that for every $w\in\bC\setminus\{1\}$ 
(the description when $w=1$ is a little complicate)
and every $\lambda\in\Lambda$, $p_n^*(\lambda,w)=0$
if and only if there exists $z_0\in\Fix^*(f_{\lambda}^n)$ 
satisfying $(f_{\lambda}^n)'(z_0)=w$ 
($p_n^*$ was introduced by Morton--Vivaldi \cite[\S 1]{MortonVivaldi95}
working on integral domains $R$ more general than $\bC$).
For the precise definition of $p_n^*(\lambda,w)$, 
see Definition \ref{th:definition}; in the case $w=0$, 
for every $n\in\bN$ and every $\lambda\in\Lambda$,
\begin{gather}
 |p_n^*(\lambda,0)|=\prod_{z\in\Fix^*(f_{\lambda}^n)}
|f_{\lambda}'(z)|.\label{eq:special}
\end{gather}

Following \cite[\S 2]{BassanelliBerteloot11}, for each $w\in\bC$, set
\begin{gather}
 \Per_f^*(n,w):=[p_n^*(\cdot,w)=0]\quad\text{on }\Lambda\label{eq:multvar}
\end{gather}
(the suffix * is added to the original notation
$\Per_f(n,w)$ in \cite[\S 2]{BassanelliBerteloot11}). 

The convergence \eqref{eq:exact} is regarded as a refinement of the following.
 
\begin{theorem}[{Bassanelli--Berteloot 
\cite[Theorem 3.1 (1)]{BassanelliBerteloot11}}]
\label{th:bifurcation}
Let $f:\Lambda\times\bP^1\to\bP^1$ 
be a holomorphic family of rational functions on $\bP^1$
of degree $d>1$ over a connected complex manifold $\Lambda$. 
Then
\begin{gather}
 \lim_{n\to\infty}\frac{\Per_f^*(n,0)}{d^n}=T_f\quad\text{as currents on }\Lambda.\label{eq:bifurcation}
\end{gather}
\end{theorem}
See Section \ref{sec:bifurcation} for the deduction of Theorem \ref{th:bifurcation}
from Theorem \ref{th:periodicbif}. This proof of Theorem \ref{th:bifurcation}
is simpler than Bassanelli--Berteloot's one,
which relies on an approximation formula of the
Lyapunov exponent of a rational function by the multipliers of
its repelling periodic points.

\begin{remark} 
 The full statement of
 Bassanelli--Berteloot \cite[Theorem 3.1]{BassanelliBerteloot11} can be
 deduced from Theorem \ref{th:bifurcation} (see also Bassanelli--Berteloot 
 \cite[\S 3]{BassanelliBerteloot09}).
 For further studies, 
 see also Buff--Gauthier \cite{BG13} and Gauthier \cite{Gauthier13}.
\end{remark}

\subsection{Organization of this article}
In Section \ref{sec:reduction}, we recall a reduction \eqref{eq:potential}
of \eqref{eq:Levin} in Theorem \ref{th:periodicbif}
as in \cite[Proof of Theorem 4.2]{DujardinFavre08}, and 
in Section \ref{sec:proofactive},
we show a dynamical counterpart of \eqref{eq:potential}.
In Section \ref{sec:analytic}, we recall a local description
of $\Per_c^*(n)$, a global decomposition of $\Per_c(n)$,
and the definition of $p_n^*(\lambda,w)$.
In Section \ref{sec:active}, we show Theorem \ref{th:periodicbif}
based on this dynamical counterpart
({\itshape plowing in the dynamical space
and reaping in the parameter space}; see, e.g., \cite[\S 1.1]{BH92}). 
In Section \ref{sec:bifurcation},
we establish a local decomposition of $\Per_f^*(n,0)$ and
show Theorem \ref{th:bifurcation} using Theorem \ref{th:periodicbif}.

\section{A reduction of Theorem \ref{th:periodicbif}}
\label{sec:reduction}
  
\begin{notation}\label{th:notation}
 As in Section \ref{sec:intro}, let $\omega$ be the
 normalized Fubini-Study area element on $\bP^1$.
 Let $\|\cdot\|$ be the Euclidean norm on $\bC^2$. 
 The origin of $\bC^2$ is also denoted by $0$, and 
 $\pi:\bC^2\setminus\{0\}\to\bP^1$ is the canonical projection.
 Setting the wedge product $(z_0,z_1)\wedge(w_0,w_1):=z_0w_1-z_1w_0$ on $\bC^2\times\bC^2$,
 the normalized chordal metric $[z,w]$ on $\bP^1$ is the function
 \begin{gather}
 (z,w)\mapsto [z,w]:=|p\wedge q|/(\|p\|\cdot\|q\|)(\le 1)\label{eq:defchordal}
 \end{gather}
 on $\bP^1\times\bP^1$, where $p\in\pi^{-1}(z),q\in\pi^{-1}(w)$.
 We normalize $\rd\rd^c$ as $\rd=\partial+\overline{\partial}$ and
 $\rd^c=i(\overline{\partial}-\partial)/(2\pi)$. Then
 $\pi^*\omega=\rd\rd^c\log\|\cdot\|$ as currents on $\bC^2\setminus\{0\}$.
\end{notation}


Let $f:\Lambda\times\bP^1\to\bP^1$ 
be a holomorphic family of rational functions on $\bP^1$
of degree $d>1$ over a connected complex manifold $\Lambda$ having a marked critical point $c:\Lambda\to\bP^1$. Recall that
$F_n(\lambda):=f_\lambda^n(c(\lambda))$ on $\Lambda$ for each $n\in\bN$. 

The following reduction of \eqref{eq:Levin} in Theorem $\ref{th:periodicbif}$
is due to Dujardin--Favre.

\begin{lemma}[{\cite[in Proof of Theorem 4.2]{DujardinFavre08}}]
\label{th:potential}
Let $f,c$, and $F_n$ be as in the above.
Then the convergence \eqref{eq:Levin} in Theorem $\ref{th:periodicbif}$ holds if
\begin{gather}
 \lim_{n\to\infty}\frac{\log[F_n,c]}{d^n+1}=0\quad\text{in }L^1_{\loc}(\Lambda).
\tag{\ref{eq:Levin}'}\label{eq:potential}
\end{gather}
\end{lemma}

Let us see Lemma \ref{th:potential}.
For every point $\lambda_0\in\Lambda$ and every open and connected neighborhood
$U$ of $\lambda_0$ in $\Lambda$ small enough, 
there is a lift $\tilde{c}:U\to\bC^2\setminus\{0\}$ of $c$
in that $\tilde{c}$ is holomorphic and that $\pi\circ\tilde{c}=c$ on $U$, and 
there is a {\itshape lift} 
$\tilde{f}:U\times\bC^2\to\bC^2$ of $f$ 
in that $\tilde{f}$ is holomorphic and that
for every $\lambda\in U$, 
$\tilde{f}_{\lambda}:=\tilde{f}(\lambda,\cdot)$
is a homogeneous polynomial endomorphism on $\bC^2$ satisfying
$\pi\circ\tilde{f}_{\lambda}=f_{\lambda}\circ\pi$ on 
$\bC^2\setminus\{0\}$ and $\tilde{f}_{\lambda}^{-1}(0)=\{0\}$.
For each $n\in\bN$, set 
\begin{gather*}
 \tilde{F}_n(\lambda):=\tilde{f}_{\lambda}^n(\tilde{c}(\lambda))\quad\text{on }U.
\end{gather*}

Recall the definition of $T_c$ (Definition \ref{th:active})
and that $F_n^*\omega=\rd\rd^c\log\|\tilde{F}_n\|$
as currents on $U$. In particular, 
\begin{gather} 
 \lim_{n\to\infty}\rd\rd^c\frac{\log\|\tilde{F}_n\|}{d^n}=T_c
\quad\text{as currents on }U
\label{eq:activelift}
\end{gather}
(see, e.g., \cite[Lemma 3.2.7]{BertelootCIME}).
By Remark \ref{th:trivialcase}, we can assume that
$F_n\not\equiv c$ on $\Lambda$ for every $n\in\bN$. 
Then for every $n\in\bN$, by the Poincar\'e-Lelong formula, 
\begin{gather}
\rd\rd^c\log|\tilde{F}_n\wedge\tilde{c}|=[\tilde{F_n}\wedge\tilde{c}=0]
=[F_n=c]=:\Per_c(n)\label{eq:PL}
\end{gather}
as currents on $U$. By \eqref{eq:defchordal}, for every $n\in\bN$,
\begin{gather}
\log|\tilde{F}_n\wedge\tilde{c}|
= \log[F_n,c]+\log\|\tilde{F}_n\|+\log\|\tilde{c}\|
\quad\text{on }U,\label{eq:chordal}
 \end{gather}
so that the continuity of $\rd\rd^c$ on $L^1_{\loc}(\Lambda)$
completes the proof of Lemma \ref{th:potential}.
 
We also recall that the {\itshape dynamical Green function} of $\tilde{f}$ 
is the local uniform limit
\begin{gather}
G^{\lambda}(p)
:=\lim_{n\to\infty}\frac{\log\|\tilde{f}_{\lambda}^n(p)\|}{d^n}
\label{eq:escaping}
\end{gather}
on $U\times(\bC^2\setminus\{0\})$
(see, e.g., \cite[Proposition 1.2]{BassanelliBerteloot07}).
In particular,
\begin{gather}
 \lim_{n\to\infty}\frac{\log\|\tilde{F}_n(\lambda)\|}{d^n}
=G^{\lambda}(\tilde{c}(\lambda)) 
\quad\text{locally uniformly on }U.\label{eq:Green}
\end{gather} 

\begin{remark}
\label{th:detail}
The locally uniform convergence \eqref{eq:Green} 
implies not only \eqref{eq:activity} but also
$T_c=\rd\rd^c_{\lambda}G^{\lambda}(\tilde{c}(\lambda))$ on $U$, which 
with DeMarco's formula
$L(f_{\lambda})=-\log d+\sum_{j=1}^{2d-2}G^{\lambda}(\tilde{c}_j(\lambda))
-(2/d)\log|\Res(\tilde{f}_{\lambda})|$,
where $\Res(\tilde{f}_{\lambda})$ is the homogeneous resultant of $\tilde{f}_{\lambda}$, on $U$ implies \eqref{eq:DeMarco}.
\end{remark} 

\section{A dynamical counterpart of \eqref{eq:potential}}\label{sec:proofactive}

For the details of complex dynamics, see, e.g., \cite{Milnor3rd}. 

\begin{definition}
 Let $f$ be a rational function on $\bP^1$. 
 The Julia set of $f$ is defined by
 $J(f):=\{z\in\bP^1:\{f^n:n\in\bN\}\text{ is not normal at }z\}$,
 whose complement in $\bP^1$ is called the Fatou set of $f$ and denoted by $F(f)$.
\end{definition}

The following is a dynamical counterpart of \eqref{eq:potential}.

\begin{lemma}\label{th:Fatou}
 Let $f$ be a rational function on $\bP^1$ of degree $d>1$.
 If a critical point $c$ of $f$ is not periodic under $f$, then 
\begin{gather}
 \lim_{n\to\infty}\frac{\log[f^n(c),c]}{d^n+1}=0.\label{eq:Fatou}
\end{gather}
\end{lemma}

\begin{proof}
 Suppose that $c\in F(f)$ 
 and contrary that \eqref{eq:Fatou} does not hold, i.e.,
 \begin{gather}
  \liminf_{n\to\infty}\frac{\log[f^n(c),c]}{d^n+1}<0.\label{eq:superatt}
 \end{gather}
 Then the Fatou component $U$ containing $c$ must intersect $f^n(U)$
 for some $n\in\bN$, so that $U$ is a cyclic Fatou component of $f$
 having, say, the period $m\in\bN$. 
 By the local non-injectivity of $f$ at $c$, $f^m:U\to U$ is not
 univalent. Then by the Denjoy--Wolff theorem (and the hyperbolicity of
 $U$, cf.\ \cite[\S 5 and \S 16]{Milnor3rd}),
 \eqref{eq:superatt} even implies that
 $c$ is a (super)attracting periodic point of $f$, which contradicts
 the non-periodicity assumption on $c$ under $f$. Hence \eqref{eq:Fatou} holds
 in this case.

 Suppose next that $c\in J(f)$. Then \eqref{eq:Fatou} follows from (the proof of) 
 Przytycki \cite[Lemma 1]{Przytycki93}, which asserts that
 {\itshape for every critical point $c\in J(f)$ of $f$ and every $n\in\bN$, 
 $[f^n(c),c]\ge 1/(20L^n)$,
 where  $L>1$ is a Lipschitz constant of $f:\bP^1\to\bP^1$ 
 with respect to the normalized
 chordal metric $[z,w]$ on $\bP^1$.} Now the proof of \eqref{eq:Fatou} is complete.
\end{proof}

\section{On $\Per_c^*(n)$, $\Per_c(n)$, and $p_n^*(\lambda,w)$}
\label{sec:analytic} 

We begin with a notion from the number theory;
see, e.g., \cite[Chapter 2]{Apostol}. 

\begin{definition}\label{th:Mobius}
 The M\"obius function $\mu:\bN\to\{0,\pm 1\}$ is defined by
 $\mu(1)=1$ and, for every $n\ge 2$, 
 by $\mu(n)=0$ if $p^2|n$ for some prime number $p$,
 and $\mu(n)=(-1)^\ell$ if $n$ factors as a product of distinct $\ell$ prime
 numbers.
\end{definition} 

Let $f:\Lambda\times\bP^1\to\bP^1$ 
be a holomorphic family of rational functions on $\bP^1$
of degree $d>1$ over a connected complex manifold $\Lambda$.

\begin{definition}[a periodic point having the formally exact period]\label{th:formallyexact}
Fix $\lambda\in\Lambda$ and $n\in\bN$.
A fixed point $w\in\bP^1$ of $f_{\lambda}^n$ is
a periodic point of $f_{\lambda}$ having the {\itshape formally exact} period $n$ if
either 
\begin{enumerate}
 \item 
$w\in\Fix^*(f_{\lambda}^n)$ or 
 \item there is a 
$m\in\bN$ satisfying $m|n$ and $m<n$ such that
 $w\in\Fix^*(f_{\lambda}^m)$  
and that
 $(f_{\lambda}^m)'(w)$ is a primitive $(n/m)$-th root of unity.
\end{enumerate}
Let $\Fix^{**}(f_{\lambda}^n)$
be the set of all periodic points of $f_{\lambda}$ having
the formally exact period $n$.
\end{definition}

\begin{remark}\label{th:disjoint}
For every distinct $n,m\in\bN$, 
$\Fix^*(f_{\lambda}^n)\cap\Fix^*(f_{\lambda}^m)=\emptyset$,
but $\Fix^{**}(f_{\lambda}^n)\cap\Fix^{**}(f_{\lambda}^m)$ might be
non-empty. 
\end{remark}

For every $\lambda_0\in\Lambda$, choose an open and connected neighborhood $U$
of $\lambda_0$ in $\Lambda$ so small
that there is a lift $\tilde{f}:U\times\bC^2\to\bC^2$ of $f$,
and set $\tilde{f}_{\lambda}=\tilde{f}(\lambda,\cdot)$, as before.

\subsection{Fundamental facts}
For the proof of the following facts, see e.g.\ Silverman
\cite[Theorem 4.5]{SilvermanDynamics} and Berteloot \cite[\S 2.3.2]{BertelootCIME}.

\begin{fact}[holomorphic family of dynatomic polynomials]\label{th:dynatomic}
 For every $n\in\bN$, 
the function
 \begin{gather}
 \Phi^*_{\tilde{f},n}(\lambda,p)
 :=\prod_{m\in\bN:\, m|n}(\tilde{f}_{\lambda}^m(p)\wedge p)^{\mu(n/m)}
\quad\text{on }U\times\bC^2\label{eq:dynatomic}
 \end{gather}
is holomorphic, and
for every $\lambda\in U$, $\Phi^*_{\tilde{f},n}(\lambda,\cdot)$ is a 
homogeneous polynomial on $\bC^2$ of degree
$\nu(n)=\sum_{m\in\bN:m|n}\mu(n/m)(d^m+1)$,
which is determined by $n$ (and $d$) and is independent of $\lambda$.
By the M\"obius inversion formula (cf.\ \cite[Chapter 2]{Apostol}),
\eqref{eq:dynatomic} is equivalent to
\begin{gather}
\tilde{f}_{\lambda}^n(p)\wedge p
=\prod_{m\in\bN:\, m|n}\Phi^*_{\tilde{f},m}(\lambda,p)
\quad\text{on }U\times\bC^2.\label{eq:inversion}
\end{gather}
\end{fact}

\begin{fact}\label{th:exactzeros}
For every $n\in\bN$ and every $\lambda\in U$, we can choose
$(\tilde{z}^{(n)}_k(\lambda))_{k=1}^{\nu(n)}$
in $\bC^2\setminus\{0\}$ such that the homogeneous polynomial
$\Phi^*_{\tilde{f},n}(\lambda,\cdot)$ factors as
\begin{gather}
 \Phi^*_{\tilde{f},n}(\lambda,p)
=\prod_{k=1}^{\nu(n)}(p\wedge\tilde{z}^{(n)}_k(\lambda))\quad\text{on }\bC^2.\label{eq:factordynatomic}
\end{gather}
Setting $z^{(n)}_k(\lambda):=\pi(\tilde{z}^{(n)}_k(\lambda))$
for each $k\in\{1,2,\ldots,\nu(n)\}$, we indeed have
\begin{gather}
\{z^{(n)}_k(\lambda):k\in\{1,\ldots,\nu(n)\}\}=\Fix^{**}(f_{\lambda}).
\label{eq:formally}
\end{gather}
Moreover, upto its permutation,
the sequence $(z^{(n)}_k(\lambda))_{k=1}^{\nu(n)}$ in $\bP^1$
is determined by $f$, $n$ and $\lambda$ and
depends on choices of neither
$\tilde{f}$ nor $(\tilde{z}^{(n)}_k(\lambda))_{k=1}^{\nu(n)}$. 
\end{fact}

\subsection{Local description of $\Per_c^*(n)$ and a global
decomposition of $\Per_c(n)$}\label{subsec:local}
In addition to $\tilde{f}$,
for every marked critical point $c:\Lambda\to\bP^1$ of $f$,
decreasing $U$ if necessary, there is also
a lift $\tilde{c}:U\to\bC^2\setminus\{0\}$ of $c$. 
For every $n\in\bN$, recall that
$\tilde{F}_n(\lambda):=\tilde{f}_{\lambda}^n(\tilde{c}(\lambda))$ on $U$,
and define the function
\begin{gather}
 \tilde{H}_n(\lambda)=\tilde{H}_{\tilde{f},n}^{\tilde{c}}(\lambda):=\Phi^*_{\tilde{f},n}(\lambda,\tilde{c}(\lambda))
\quad\text{on }U,\label{eq:primitivepotential}
\end{gather}
which is holomorphic by Fact \ref{th:dynatomic}.
Then by \eqref{eq:factordynatomic},
\begin{gather}
 \tilde{H}_n=\prod_{k=1}^{\nu(n)}(\tilde{c}\wedge\tilde{z}^{(n)}_k)\quad\text{on }U,\label{eq:explicit}
\end{gather} 
and by \eqref{eq:inversion},
\begin{gather}
 \tilde{F}_n\wedge\tilde{c}=\prod_{m\in\bN:\, m|n}\tilde{H}_m
 \quad\text{on }U.\label{eq:decomositionpotential} 
\end{gather}


\begin{lemma}[a local description of $\Per_c^*(n)$]\label{th:local}
For every $n\in\bN$, 
\begin{gather} 
 \Per_c^*(n)|U=[\tilde{H}_n=0].\label{eq:analytic}
\end{gather}
\end{lemma}

\begin{proof}
For every $n\in\bN$, we claim that
\begin{gather}
 X_n^*:=\supp(\Per_c^*(n)|U)=\supp[\tilde{H}_n=0];\label{eq:support}
\end{gather}
for,
\begin{align*}
 \supp(\Per_c^*(n)|U)
=&\{\lambda\in U:c(\lambda)\in\Fix^*(f_{\lambda}^n)\}
&(\text{by \eqref{eq:superattdivisor} and \eqref{eq:exactsuperatt}})\\
=&\{\lambda\in U:c(\lambda)\in\Fix^{**}(f_{\lambda}^n)\}
&(\text{by }(f_{\lambda}^n)'(c(\lambda))=0\neq 1)\\
=&\supp[\tilde{H}_n=0]
&(\text{by \eqref{eq:explicit} and \eqref{eq:formally}}).
\end{align*}
This also implies that for every distinct $m,n\in\bN$,
\begin{multline}
\supp[\tilde{H}_m=0]\cap\supp[\tilde{H}_n=0]\\
=\{\lambda\in U:c(\lambda)\in\Fix^*(f_{\lambda}^m)\cap\Fix^*(f_{\lambda}^n)(=\emptyset)\}=\emptyset.\label{eq:empty}
\end{multline}
Fix $n\in\bN$. We claim that
for every irreducible component $V$ of $X_n^*$,
$\ord_V(\Per_c^*(n)|U)=\ord_V[\tilde{H}_n=0]$; for,
\begin{align*}
 \ord_V(\Per_c^*(n)|U)
=&\ord_V(\Per_c(n)|U) 
& (\text{by \eqref{eq:exactsuperatt}})\\
=&\ord_V[\tilde{F}_n\wedge\tilde{c}=0]
&(\text{by \eqref{eq:PL}})\\
=&\ord_V[\tilde{H}_n=0]
&(\text{by \eqref{eq:decomositionpotential}, \eqref{eq:support}, and \eqref{eq:empty}}).
\end{align*}
Now the proof is complete.
\end{proof}

For every $n\in\bN$, recall also that
$F_n(\lambda):=f_{\lambda}^n(c(\lambda))$ on $\Lambda$.

\begin{lemma}[a global decomposition of $\Per_c(n)$]\label{th:global}
For every $n\in\bN$, under the assumption that $F_n\not\equiv c$ on $\Lambda$,
it holds that
\begin{gather}
 \Per_c(n)=\sum_{m\in\bN:\, m|n}\Per_c^*(m)\quad\text{on }\Lambda.\label{eq:decomposition}
\end{gather}
\end{lemma}

\begin{proof}
Fix $n\in\bN$. By \eqref{eq:PL}, \eqref{eq:decomositionpotential}, 
and \eqref{eq:analytic}, if $F_n\not\equiv c$ on $U$, then 
$\Per_c(n)|U=\sum_{m\in\bN:\, m|n}\Per_c^*(m)|U$, so
\eqref{eq:decomposition} holds since $\lambda_0$ is arbitrary. 
\end{proof}

\subsection{The definition of $p_n^*(\lambda,w)$}\label{th:multiplier}
For the details, see Berteloot \cite[\S 2.3.1]{BertelootCIME}.

For every $n\in\bN$, every $\lambda\in\Lambda$,
and every $j\in\{0,1,2,\ldots,\nu(n)\}$,
let $\sigma_j^*(n,\lambda)$
be the $j$-th elementary symmetric function associated to 
$((f_{\lambda}^n)'(z^{(n)}_k(\lambda)))_{k=1}^{\nu(n)}$.

Then, for every $n\in\bN$, by 
the holomorphy of $\Phi^*_{\tilde{f},n}$ and $f$,
 the function $\sigma_j^*(n,\cdot)$ is holomorphic on $\Lambda$
 for every $j\in\{0,1,2,\ldots,\nu(n)\}$.

\begin{definition}[cf.\ {\cite[\S 2.1]{BassanelliBerteloot11}}]\label{th:definition}
 For every $n\in\bN$, there is a holomorphic function 
 $p_n^*(\lambda,w)=p_{f,n}^*(\lambda,w)$ on $\Lambda\times\bC$,
 which is unique up to multiplication in $n$-th roots of unity,
 such that
\begin{gather} 
 (p_n^*(\lambda,w))^n=\sum_{j=0}^{\nu(n)}\sigma_j^*(n,\lambda)(-w)^{\nu(n)-j}
\quad\text{on }\Lambda\times\bC.\label{eq:algebraic}
\end{gather}
\end{definition}  

For every $n\in\bN$ and every $\lambda\in\Lambda$,
we have
 \begin{multline*}
 |p_n^*(\lambda,0)|\\
=|(\sigma_{\nu(n)}^*(n,\lambda))^{1/n}|
=\prod_{k=1}^{\nu(n)}|(f_{\lambda}^n)'(z^{(n)}_k(\lambda))|^{1/n}
=\prod_{z\in\Fix^*(f_{\lambda}^n)}|(f_{\lambda}^n)'(z)|^{1/n},
 \end{multline*}
where the final equality holds since 
for every $k\in\{1,2,\ldots,\nu(n)\}$ satisfying 
$|(f_{\lambda}^n)'(z^{(n)}_k(\lambda))|\neq 1$,
$z^{(n)}_k(\lambda)$ is in $\Fix^*(f_{\lambda}^n)$ and is a simple root of
$\Phi^*_{\tilde{f},n}(\lambda,\cdot)$.

Hence, by the chain rule, we have
not only \eqref{eq:special} but also
\begin{gather}
 |p_n^*(\lambda,0)|=\prod_{k=1}^{\nu(n)}|f_{\lambda}'(z^{(n)}_k(\lambda))|.\tag{\ref{eq:special}'} \label{eq:productmult}
\end{gather}

\section{Proof of Theorem \ref{th:periodicbif}}
\label{sec:active}

\subsection{Basic facts}
Following Bassanelli--Berteloot \cite[Theorem 2.5]{BassanelliBerteloot11},
we refer to the following as a {\itshape compactness
principle} for subharmonic functions.

\begin{theorem}[{\cite[a consequence of Theorem 4.1.9 (a)]{Hormander83}}]\label{th:compactness}
Let $(\phi_j)$ be a sequence of subharmonic functions on
a domain $U$ in $\bR^n$, and suppose that $(\phi_j)$ is locally
uniformly bounded from above.
If $\phi:=\lim_{j\to\infty}\phi_j$ exists Lebesgue a.e.\ on 
$U$, then indeed $\lim_{j\to\infty}\phi_j=\phi$ in $L^1_{\loc}(U)$.
\end{theorem}


\subsection{Proof of Theorem \ref{th:periodicbif}}
 Let $f$ and $c$ be as in Theorem \ref{th:periodicbif}, and
 recall 
 that $F_n(\lambda):=f_{\lambda}^n(c(\lambda))$ 
 on $\Lambda$ for each $n\in\bN$.
 By Remark \ref{th:trivialcase}, we can assume 
 without loss of generality that
 \begin{gather}
 F_n\not\equiv c\quad\text{on }\Lambda\quad\text{for every }n\in\bN.\label{eq:nonidentical} 
 \end{gather}

 For every $\lambda_0\in\Lambda$, choose an open and connected neighborhood $U$
 of $\lambda_0$ in $\Lambda$ so small
 that there are a lift $\tilde{f}:U\times\bC^2\to\bC^2$ 
 of $f$ and a lift $\tilde{c}:U\to\bC^2\setminus\{0\}$ of $c$. 
Recall that $\tilde{f}_{\lambda}=\tilde{f}(\lambda,\cdot)$, that
for every $n\in\bN$, 
$\tilde{F}_n(\lambda):=\tilde{f}_{\lambda}^n(\tilde{c}(\lambda))$ on $U$
and 
$\log|\tilde{F}_n\wedge\tilde{c}|
= \log[F_n,c]+\log\|\tilde{F}_n\|+\log\|\tilde{c}\|$ on $U$
(\eqref{eq:chordal}),
and that 
$\lim_{n\to\infty}(\log\|\tilde{F}_n(\lambda)\|)/d^n=G^{\lambda}(\tilde{c}(\lambda))$ locally uniformly on $U$ (\eqref{eq:Green}), 
in Section \ref{sec:reduction}.

 Let us first prove \eqref{eq:Levin} and then prove \eqref{eq:exact}.

\begin{proof}[Proof of $\eqref{eq:Levin}$]
 According to Lemma \ref{th:potential}, it is sufficient to prove \eqref{eq:potential}. Let us show \eqref{eq:potential}.
 By \eqref{eq:Green} and \eqref{eq:chordal}, 
 the sequence $((\log|\tilde{F}_n\wedge\tilde{c}|)/(d^n+1))$ of 
 plurisubharmonic functions on $U$ is locally uniformly bounded from above on $U$. 
 \begin{claim}
 \begin{gather*}
 \lim_{n\to\infty}\frac{\log|\tilde{F}_n(\lambda)\wedge\tilde{c}(\lambda)|}{d^n+1}
 =G^{\lambda}(\tilde{c}(\lambda))
 \quad\text{for Lebesgue a.e. } \lambda\in U.
 \end{gather*} 
 \end{claim}
 \begin{proof}
 By the assumption \eqref{eq:nonidentical},
 the union $\bigcup_{n\in\bN}\supp[F_n=c]$
 is a Lebesgue null subset in $\Lambda$,
 and by \eqref{eq:chordal}, Lemma \ref{th:Fatou}, and \eqref{eq:Green},
 for every $\lambda\in U\setminus\bigcup_{n\in\bN}\supp[F_n=c]$,
 \begin{multline*}
 \lim_{n\to\infty}\frac{\log|\tilde{F}_n(\lambda)\wedge\tilde{c}(\lambda)|}{d^n+1}
 =\lim_{n\to\infty}\frac{\log[f_{\lambda}^n(c(\lambda)),c(\lambda)]}{d^n+1}
 +\lim_{n\to\infty}\frac{\log\|\tilde{F}_n(\lambda)\|}{d^n}\\
 =0+G^{\lambda}(\tilde{c}(\lambda))
 =G^{\lambda}(\tilde{c}(\lambda)).
 \end{multline*}
 This completes the proof.
 \end{proof}

 By this claim and Theorem \ref{th:compactness} (a compactness principle),
 using also \eqref{eq:Green} and \eqref{eq:chordal}, we have
 \begin{multline*}
 \lim_{n\to\infty}\frac{\log[F_n(\lambda),c(\lambda)]}{d^n+1}
 =\lim_{n\to\infty}\frac{\log|\tilde{F}_n(\lambda)\wedge\tilde{c}(\lambda)|}{d^n+1}-\lim_{n\to\infty}\frac{\log\|\tilde{F}_n(\lambda)\|}{d^n}\\
 =G^{\lambda}(\tilde{c}(\lambda))-G^{\lambda}(\tilde{c}(\lambda))=0
 \quad\text{in }L^1_{\loc}(U).
 \end{multline*}
 Since $\lambda_0$ is arbitrary, 
 the proof of \eqref{eq:potential}, so of \eqref{eq:Levin}, is complete.
\end{proof}

\begin{proof}[Proof of \eqref{eq:exact}]
%
%
%

Under the assumption \eqref{eq:nonidentical},
by the M\"obius inversion of 
the global decomposition \eqref{eq:decomposition} of $\Per_c(n)$
(in Lemma \ref{th:global}), 
for every smooth $(\dim_{\bC}\Lambda-1,\dim_{\bC}\Lambda-1)$-form
$\phi$ on $\Lambda$,
 \begin{multline*}
  \left|\left\langle\phi,\Per_c^*(n)-\Per_c(n)\right\rangle\right|\\
\le\sum_{m\in\bN:\, m|n\text{ and }m<n}
  \left|\mu\left(\frac{n}{m}\right)\right|\cdot
|\langle\phi,\Per_c(m)\rangle|
 =O(d^{n/2})\quad\text{as }n\to\infty, 
\end{multline*} 
where the final order estimate follows from
\eqref{eq:Levin} and
$m\le n/2$ for every $m\in\bN$ satisfying $m|n$ and $m<n$.

Hence \eqref{eq:Levin} implies \eqref{eq:exact}.
\end{proof} 

\section{Proof of Theorem \ref{th:bifurcation}}\label{sec:bifurcation}
Let $f:\Lambda\times\bP^1\to\bP^1$ be as in Theorem \ref{th:bifurcation}.
 Taking a finitely-sheeted
 possibly ramified covering of $\Lambda$ if necessary, we assume 
 without loss of generality that
 there are marked critical points $c_1,\ldots,c_{2d-2}:\Lambda\to\bP^1$ of $f$ 
 such that for every $\lambda\in\Lambda$, $c_1(\lambda),\ldots, c_{2d-2}(\lambda)$
 are all the critical points of $f_{\lambda}$,
 taking into account their multiplicities. 
 
\begin{lemma}[a local decomposition of $\Per_f^*(n,0)$]\label{th:identity}
For every $\lambda_0\in\Lambda$, there are an 
open neighborhood $U$ of $\lambda_0$ in $\Lambda$ and $N_0\in\bN$
such that for every $n>N_0$,
\begin{gather}
 \Per_f^*(n,0)=\sum_{j=1}^{2d-2}\Per_{c_j}^*(n)\quad\text{on }U.
\label{eq:bifactive}
\end{gather} 
\end{lemma}

By Lemma \ref{th:identity} and \eqref{eq:DeMarco}, 
the convergence \eqref{eq:exact} in Theorem \ref{th:periodicbif} implies
\begin{gather*}
 \lim_{n\to\infty}\frac{\Per_f^*(n,0)}{d^n+1}
 =\lim_{n\to\infty}\sum_{j=1}^{2d-2}\frac{\Per_{c_j}^*(n)}{d^n+1}
 =\sum_{j=1}^{2d-2}T_{c_j}=T_f\quad\text{on }U.
\end{gather*}
Since $\lambda_0$ is arbitrary, 
the convergence \eqref{eq:bifurcation} in Theorem \ref{th:bifurcation} holds.

\begin{remark}
  For every $n\in\bN$ and every $\lambda\in\Lambda$, set 
 $R^*(f_{\lambda}^n):=\{w\in\Fix^*(f_{\lambda}^n):|(f_{\lambda}^n)'(w)|>1\}$. 
 The original proof of Theorem \ref{th:bifurcation} is based on
 the approximation
 \begin{gather*}
 L(f_{\lambda})=\lim_{n\to\infty}\frac{1}{nd^n}\sum_{z\in R^*(f_{\lambda}^n)}\log|(f_{\lambda}^n)'(z)|\quad\text{for each }\lambda\in\Lambda;
 \end{gather*} 
 for the details of this formula,
 see Berteloot--Dupont--Molino 
 \cite[Corollary 1.6]{BDM08}, and also \cite{BertelootLyapunov,OkuLog}.
 The proof of Theorem \ref{th:bifurcation}
 presented here does not rely on this approximation and, moreover,
 the argument developed in the proof of Lemma \ref{th:identity},
 combined with the proof of Theorem \ref{th:periodicbif},
 is simpler than the original one.
\end{remark}

\begin{proof}[Proof of Lemma $\ref{th:identity}$]
 Fix $\lambda_0\in\Lambda$. 
 Choosing an open and connected neighborhood $U$
 of $\lambda_0$ in $\Lambda$ small enough,
 we have a lift 
 $\tilde{f}:U\times\bC^2\to\bC^2$ of $f$ and
 a lift $\tilde{c}_j:U\to\bC^2\setminus\{0\}$ of $c_j$
 for every $j\in\{1,2,\ldots,2d-2\}$ normalized so that
 for every $\lambda\in U$, the Jacobian determinant 
 of $\tilde{f}_\lambda=\tilde{f}(\lambda,\cdot)$ factors as
 \begin{gather}
 (\det D\tilde{f}_\lambda)(p)=\prod_{j=1}^{2d-2}(p\wedge\tilde{c}_j(\lambda))
  \quad\text{on }\bC^2.\label{eq:Jacobian}
 \end{gather}  
For each $n\in\bN$, recall the definition \eqref{eq:primitivepotential}
of the function $\tilde{H}^{\tilde{c}_j}_{\tilde{f},n}$ on $U$ and set
\begin{gather*}
 \tilde{H}_n^{(j)}:=\tilde{H}_{\tilde{f},n}^{\tilde{c}_j}\quad\text{on }U,
\text{ for each }j\in\{1,2,\ldots,2d-2\}. 
\end{gather*}
For each $n\in\bN$ and each $\lambda\in U$,
recall also the definition of
$(\tilde{z}^{(n)}_k(\lambda))_{k=1}^{\nu(n)}$ in $\bC^2\setminus\{0\}$ 
and that $z^{(n)}_k(\lambda)
=\pi(\tilde{z}^{(n)}_k(\lambda))$ for each $k\in\{1,2,\ldots,\nu(n)\}$,
in Fact \ref{th:exactzeros}. 
 
\begin{claim}
 For every $n\in\bN$, 
 $|p_n^*(\cdot,0)|=|\prod_{j=1}^{2d-2}\tilde{H}_n^{(j)}|\cdot e^{r_n}$ on $U$.

 Here $r_n(\lambda):=-\nu(n)\log d
 +2(\sum_{k=1}^{\nu(n)}\log\|\tilde{f}_{\lambda}(\tilde{z}^{(n)}_k(\lambda))\|
 -\sum_{k=1}^{\nu(n)}\log\|\tilde{z}^{(n)}_k(\lambda)\|)$
 is a pointwise finite function on $U$.
\end{claim}

 \begin{proof}
 For every $\lambda\in U$ and every $n\in\bN$, 
 by a computation involving Euler's identity 
 (cf.\ \cite[Theorem 4.3]{Jonsson98}), for every $k\in\{1,2,\ldots,\nu(n)\}$, 
 \begin{gather*}
 |f_{\lambda}'(z^{(n)}_k(\lambda))|
 =\frac{1}{d}\frac{\|\tilde{f}_{\lambda}(\tilde{z}^{(n)}_k(\lambda))\|^2}{\|\tilde{z}^{(n)}_k(\lambda))\|^2}
 |(\det(D\tilde{f}_{\lambda})(\tilde{z}^{(n)}_k(\lambda))|,
 \end{gather*}
 so by \eqref{eq:productmult}, we have
 $|p_n^*(\lambda,0)|=|\prod_{k=1}^{\nu(n)}
(\det D\tilde{f}_{\lambda})(\tilde{z}^{(n)}_k(\lambda))|\cdot e^{r_n(\lambda)}$.
 Moreover, for every $\lambda\in U$ and every $n\in\bN$,
 by \eqref{eq:Jacobian} and \eqref{eq:explicit}, we have
 \begin{gather*}
  \prod_{k=1}^{\nu(n)}(\det D\tilde{f}_{\lambda})(\tilde{z}^{(n)}_k(\lambda))
  =\prod_{k=1}^{\nu(n)}\prod_{j=1}^{2d-2}(\tilde{z}^{(n)}_k(\lambda)\wedge\tilde{c}_j(\lambda))
 =\prod_{j=1}^{2d-2}\tilde{H}_n^{(j)}(\lambda),
 \end{gather*}
 which completes the proof.
 \end{proof}

 Under the convention $\min\emptyset=0$, set
 \begin{gather*}
 N_0:=\max_{j\in\{1,2,\ldots,2d-2\}}
 \left(\min\{n\in\bN:f_{\lambda_0}^n(c_j(\lambda_0))=c_j(\lambda_0)\}\right)\in\bN\cup\{0\}.
 \end{gather*}
 For every $n>N_0$, 
 neither $\tilde{H}_n^{(j)}$
 for every $j\in\{1,2,\ldots,2d-2\}$
 nor  
$p_n^*(\cdot,0)$ identically vanish on $U$, and
 by Claim, $p_n^*(\cdot,0)/(\prod_{j=1}^{2d-2}\tilde{H}_n^{(j)})$ 
 has neither zeros nor poles on $U$.
 Hence, for every $n>N_0$, by the Poincar\'e-Lelong formula, we have
 $[p_n^*(\cdot,0)=0]=\sum_{j=1}^{2d-2}[\tilde{H}_n^{(j)}=0]$ on $U$,
 so by the definition \eqref{eq:multvar} of $\Per_f^*(n,0)$ and
 the local description \eqref{eq:analytic}
 of $\Per_{c_j}^*(n)$ (in Lemma \ref{th:local}), 
 we have \eqref{eq:bifactive}. Now the proof of Lemma \ref{th:identity}
 is complete.
\end{proof}
 
\begin{acknowledgement}
 The author thanks the referee for a very careful scrutiny and
 invaluable comments. This research is 
partially supported by JSPS Grant-in-Aid for Young Scientists (B), 24740087.
\end{acknowledgement}

\def\cprime{$'$}


\begin{thebibliography}{10}
 
\bibitem{Apostol}
{\sc Apostol,~T.~M.} {\em Introduction to analytic number theory}, Springer
  (1976).

\bibitem{BassanelliBerteloot07}
{\sc Bassanelli,~G.{\rm\ and }Berteloot,~F.} Bifurcation currents in
  holomorphic dynamics on {$\Bbb P\sp k$}, {\em J. Reine Angew. Math.}, {\bf
  608} (2007), 201--235.

\bibitem{BassanelliBerteloot09}
{\sc Bassanelli,~G.{\rm\ and }Berteloot,~F.} Lyapunov exponents, bifurcation
  currents and laminations in bifurcation loci, {\em Mathematische Annalen},
  {\bf 345}, 1 (2009), 1--23.

\bibitem{BassanelliBerteloot11}
{\sc Bassanelli,~G.{\rm\ and }Berteloot,~F.} Distribution of polynomials with
  cycles of a given multiplier, {\em Nagoya Mathematical Journal}, {\bf 201}
  (2011), 23--43.

\bibitem{BertelootLyapunov}
{\sc Berteloot,~F.} Lyapunov exponent of a rational map and multipliers of
  repelling cycles, {\em Riv. Mat. Univ. Parma (N.S.)}, {\bf 1}, 2 (2010),
  263--269.

\bibitem{BertelootCIME}
{\sc Berteloot,~F.} Bifurcation currents in holomorphic families of rational
  maps, Pluripotential Theory, Springer (2013),  1--93.

\bibitem{BDM08}
{\sc Berteloot,~F., Dupont,~C.{\rm\ and }Molino,~L.} Normalization of bundle
  holomorphic contractions and applications to dynamics, {\em Ann. Inst.
  Fourier (Grenoble)}, {\bf 58}, 6 (2008), 2137--2168.

\bibitem{BH92}
{\sc Branner,~B.{\rm\ and }Hubbard,~J.~H.} The iteration of cubic polynomials
  Part II: Patterns and parapatterns, {\em Acta mathematica}, {\bf 169}, 1
  (1992), 229--325.

\bibitem{BG13}
{\sc {Buff},~X.{\rm\ and }{Gauthier},~T.} {Quadratic polynomials, multipliers
  and equidistribution}, {\em Proc. Amer. Math. Soc.} (to appear).

\bibitem{Demaillybook}
{\sc Demailly,~J.-P.} Complex analytic and algebraic geometry, {\em available
  at http://www-fourier.ujf-grenoble.fr/\~{}demailly/manuscripts/agbook.pdf}
  (2012).

\bibitem{DeMarco03}
{\sc DeMarco,~L.} Dynamics of rational maps: {L}yapunov exponents,
  bifurcations, and capacity, {\em Math. Ann.}, {\bf 326}, 1 (2003), 43--73.

\bibitem{DSsurvey}
{\sc Dinh,~T.-C.{\rm\ and }Sibony,~N.} Dynamics in several complex variables:
  endomorphisms of projective spaces and polynomial-like mappings, Holomorphic
  dynamical systems, Vol. 1998 of {\em Lecture Notes in Math.}, Springer,
  Berlin (2010),  165--294.

\bibitem{DujardinSurvey}
{\sc {Dujardin},~R.} {Bifurcation currents and equidistribution in parameter
  space}, Frontiers in Complex Dynamics: in Celebration of John Milnor's 80th
  birthday, Princeton University Press (2014),  515--566.

\bibitem{DujardinFavre08}
{\sc Dujardin,~R.{\rm\ and }Favre,~C.} Distribution of rational maps with a
  preperiodic critical point, {\em American journal of mathematics}, {\bf 130},
  4 (2008), 979--1032.

\bibitem{Gauthier13}
{\sc {Gauthier},~T.} {Equidistribution towards the bifurcation current I :
  Mulitpliers and degree d polynomials}, {\em ArXiv e-prints} (Dec. 2013).

\bibitem{Hormander83}
{\sc H{\"o}rmander,~L.} {\em The analysis of linear partial differential
  operators. {I}}, Vol. 256 of {\em Grundlehren der Mathematischen
  Wissenschaften [Fundamental Principles of Mathematical Sciences]},
  Springer-Verlag, Berlin (1983), Distribution theory and Fourier analysis.

\bibitem{Jonsson98}
{\sc Jonsson,~M.} Sums of {L}yapunov exponents for some polynomial maps of
  {${\bf C}\sp 2$}, {\em Ergodic Theory Dynam. Systems}, {\bf 18}, 3 (1998),
  613--630.

\bibitem{Klimek91}
{\sc Klimek,~M.} {\em Pluripotential theory}, Vol.~6 of {\em London
  Mathematical Society Monographs. New Series}, The Clarendon Press Oxford
  University Press, New York (1991), Oxford Science Publications.

\bibitem{Levin90}
{\sc Levin,~G.~M.} Theory of iterations of polynomial families in the complex
  plane, {\em Journal of Mathematical Sciences}, {\bf 52}, 6 (1990),
  3512--3522.

\bibitem{Lyubich83stability}
{\sc Lyubich,~M.~Y.} Some typical properties of the dynamics of rational maps,
  {\em Russian Mathematical Surveys}, {\bf 38}, 5 (1983), 154--155.

\bibitem{MSS}
{\sc Ma{\~n}\'e,~R., Sad,~P.{\rm\ and }Sullivan,~D.} On the dynamics of
  rational maps, {\em Ann. Sc. E.N.S.,4$\grave{e}$me s$\acute{e}$rie}, {\bf 16}
  (1983), 193--217.

\bibitem{McMullen:renorm}
{\sc McMullen,~C.~T.} {\em Complex dynamics and renormalization}, Vol. 135 of
  {\em Annals of Mathematics Studies}, Princeton University Press, Princeton,
  NJ (1994).

\bibitem{Milnor3rd}
{\sc Milnor,~J.} {\em Dynamics in one complex variable}, Vol. 160 of {\em
  Annals of Mathematics Studies}, Princeton University Press, Princeton, NJ,
  third edition (2006).

\bibitem{MortonVivaldi95}
{\sc Morton,~P.{\rm\ and }Vivaldi,~F.} Bifurcations and discriminants for
  polynomial maps, {\em Nonlinearity}, {\bf 8}, 4 (1995), 571.

\bibitem{OkuLog}
{\sc Okuyama,~Y.} {Repelling periodic points and logarithmic equidistribution
  in non-archimedean dynamics.}, {\em Acta Arith.}, {\bf 152}, 3 ({2012}),
  267--277.

\bibitem{Pham05}
{\sc Pham,~N.-m.} Lyapunov exponents and bifurcation current for
  polynomial-like maps, {\em arXiv preprint math/0512557} (2005).

\bibitem{Przytycki93}
{\sc Przytycki,~F.} Lyapunov characteristic exponents are nonnegative, {\em
  Proc. Amer. Math. Soc.}, {\bf 119}, 1 (1993), 309--317.

\bibitem{SilvermanDynamics}
{\sc Silverman,~J.~H.} {\em The arithmetic of dynamical systems}, Vol. 241 of
  {\em Graduate Texts in Mathematics}, Springer, New York (2007).

\end{thebibliography}
\end{document}